\documentclass{amsart}
\usepackage[utf8x]{inputenc}
\usepackage[T1]{fontenc}
\usepackage{amsthm}
\usepackage{amsmath}
\usepackage{amssymb}
\usepackage{graphicx}
\usepackage{mathrsfs} 
\usepackage{mathtools}
\usepackage{hyperref}
\usepackage{enumitem}
\numberwithin{equation}{section}
\newtheorem{Def}{Definition}
\numberwithin{Def}{section}
\newtheorem{Lem}[Def]{Lemma}
\newtheorem{Cor}[Def]{Corollary}
\newtheorem{Thm}[Def]{Theorem}

\newtheorem{Pro}[Def]{Proposition}
\title[A uniform lower bound for the Zhang--Kawazumi invariant]{A uniform lower bound for the Zhang--Kawazumi invariant and applications to the Bogomolov conjecture}
\author{Robert Wilms}
\address{Universit\'e de Caen Normandie, CNRS, LMNO UMR 6139, F-14000 Caen, France}
\email{\href{mailto:robert.wilms87@gmail.com}{robert.wilms87@gmail.com}}
\date{\today}
\thanks{The author gratefully acknowledges funding from the \emph{Région Normandie} under the program \emph{``Normandie Recherche -- Objectif Labels d'excellence 2025''} (Project \textbf{MERCI})}
\begin{document}
	
	\begin{abstract}
		We prove that the Zhang--Kawazumi invariant \(\varphi(X)\) of a compact and connected Riemann surface \(X\) of genus \(g\ge 2\) is strictly larger than 
		\[\frac{g(g+2)-(2g+1)H_g}{g-1},\]
		where \(H_g=\sum_{k=1}^g \frac{1}{k}\) denotes the \(g\)-th harmonic number. If \(X\) is hyperelliptic, we give the stronger bound \(\varphi(X)>\frac{g}{2}(H_g-1)\). The proof relies on a new expression of \(\varphi(X)\) in terms of a quadratic form on the space of smooth Hermitian matrix-valued functions on \(X\), evaluated at certain projector matrices. As an arithmetic application, we deduce new lower bounds for the self-intersection number \(\omega_a^2\) of the admissible adelic metrized canonical bundle \(\omega_a\) of a smooth projective curve of genus \(g\ge 2\) over a number field and hence new uniform height bounds in the Bogomolov conjecture.
	\end{abstract}
	\maketitle
	\section{Introduction}
	The Zhang--Kawazumi invariant \(\varphi(X)\) of a compact and connected Riemann surface \(X\) of genus \(g\ge 2\) was introduced independently by Zhang in \cite[Theorem 1.3.1]{Zha10} and by Kawazumi in \cite[Section 1]{Kaw08} as a certain integral of the Arakelov--Green function. Zhang showed how \(\varphi(X)\) is related to the self-intersection number \(\omega_a^2\) of the admissible adelic metrized canonical bundle \(\omega_a\) of a smooth projective curve of genus \(g\ge 1\) over a number field. Earlier, Zhang \cite[Theorem 5.6]{Zha93} showed that \(\omega_a^2\) is related to the essential minimum of the Néron--Tate height on the Abel--Jacobi embedding of the curve in its Jacobian.
	
	The goal of this note is to prove a new explicit and uniform lower bound for \(\varphi(X)\) depending only on \(g\). For a precise definition of \(\varphi(X)\) we refer to Section \ref{sec_arakelov-invariants}. For every \(n\in\mathbb{Z}_{\ge 1}\), we write \(H_n=\sum_{k=1}^n\frac{1}{k}\) for the \(n\)-th harmonic number. Our main result is the following theorem.
	\begin{Thm}\label{mainthm}
		Let \(X\) be a compact and connected Riemann surface of genus \(g\ge 2\). The Zhang--Kawazumi invariant \(\varphi(X)\) satisfies
		\[\varphi(X)>\frac{g(g+2)-(2g+1)H_g}{g-1}\ge \begin{cases} \frac{1}{2}& \text{if }g=2,\\ g-2\log g& \text{if } g\ge 3.\end{cases}\]
		If \(X\) is hyperelliptic, we have the stronger bound \(\varphi(X)>\frac{g}{2}(H_g-1)\ge \frac{g}{2}\left(\log g-\frac{1}{2}\right)\).
	\end{Thm}
	Because of their relation to the essential minimum of the Néron--Tate height on a curve, lower bounds for the Zhang--Kawazumi invariant are of considerable interest in Diophantine geometry. Both Zhang \cite[Proposition 2.5.3, Remark 1]{Zha10} and Kawazumi \cite[Corollary 1.2]{Kaw08} showed that \(\varphi(X)>0\) if \(g\ge 2\). In their recent proof of a uniform and quantitative version of the Mordell conjecture, Yu, Yuan, and Zhou \cite[Theorem 9.19]{YYZ26} proved the uniform lower bound \(\varphi(X)\ge \frac{1}{15625g^{1/3}}\) for \(g\ge 2\) with slightly better bounds for \(g\in\{2,3,4\}\), using hyperbolic geometry. In contrast, our proof of Theorem \ref{mainthm} uses only the canonical Arakelov metric. Moreover, our bound grows linearly with the genus. The special value \(\varphi(X_B)=0.519860385\dots\), computed by Pioline \cite[Equation (78)]{Pio16} for the Burnside curve \(X_B\) defined by \(y^2=x(x^4-1)\), shows that the uniform lower bound \(\varphi(X)>\frac{1}{2}\) is already close to the value of \(\varphi\) at this distinguished genus-two curve. This also represents a substantial numerical improvement over the previous uniform bound \(\varphi(X)\geq1/6400\) of Yu--Yuan--Zhou in genus \(2\).
	
	In the proof, we use the positive semi-definite symmetric bilinear form
	\[\mathcal{E}\colon C^{\infty}(X,\operatorname{Herm}_g)\times C^{\infty}(X,\operatorname{Herm}_g)\to \mathbb{R},\qquad (A,B)\mapsto \frac{i}{\pi}\int_X\operatorname{tr}\left(\partial A\wedge\overline{\partial}B\right),\]
	where \(\operatorname{Herm}_g\) denotes the real vector space of Hermitian \(g\times g\) matrices.
	We will consider the osculating flag
	\[0=E_0\subsetneq E_1\subsetneq\dots\subsetneq E_g=\mathcal{O}_X^{\oplus g}\]
	of subbundles of \(\mathcal{O}_X^{\oplus g}\) and the associated orthogonal projectors \(P_k\in C^{\infty}(X,\operatorname{Herm}_g)\) onto \(E_k\) for all \(k\in\{1,\dots,g\}\). We show that the Zhang--Kawazumi invariant can be expressed by
	\[\varphi(X)=2\mathcal{E}(V,V),\]
	where \(V=\Delta^{-1}(gP_1-I_g)\) and \(\Delta\) denotes the Laplacian operator. The lower bound in Theorem \ref{mainthm} will follow from the non-negativity
	\[\mathcal{E}\left(V+\sum_{k=1}^{g-1}\frac{g-k}{2\deg E_k}P_k,V+\sum_{k=1}^{g-1}\frac{g-k}{2\deg E_k}P_k\right)\ge 0\]
	and an explicit computation of the left-hand side. Moreover, one can show that this bound is strict using that \(\mathcal{E}(A,A)=0\) implies that \(A\) is constant.
	
	Next, we discuss applications of Theorem \ref{mainthm} to Arakelov theory and to the Bogomolov conjecture. Let \(C\) be a smooth projective geometrically connected curve of genus \(g\ge 2\) defined over a number field \(K\). Replacing \(K\) by a finite extension, we may assume that \(C\) has semistable reduction over \(B=\operatorname{Spec}\mathcal{O}_K\), where \(\mathcal{O}_K\) denotes the ring of integers of \(K\). We write \(\omega\) for the canonical bundle on \(C\). Zhang introduced in \cite{Zha93} a canonical admissible adelic metric on \(\omega\) and defined the self-pairing \(\omega_a^2=(\omega_a,\omega_a)\) of the corresponding adelic metrized line bundle \(\omega_a\). In \cite[Conjecture 1.4.1]{Zha10}, he conjectured that
	\begin{align}\label{equ_conjecture-zhang}
	\omega_a^2\ge \frac{2g-2}{2g+1}\left(\sum_{v\in|B|}\varphi(\Gamma_v(C))\log N(v)+\sum_{\sigma\colon K\to\mathbb{C}}\varphi(C_\sigma)\right),
	\end{align}
	where \(|B|\) denotes the closed points of \(B\), \(\Gamma_v(C)\) is the polarized metrized reduction graph of \(C\) at the place \(v\), \(\varphi(\Gamma_v(C))\) is a certain invariant of polarized metrized graphs defined in \cite[Theorem 1.3.1]{Zha10}, \(N(v)\) is the cardinality of the residue field at \(v\), the second sum runs over all embeddings from \(K\) into the complex numbers \(\mathbb{C}\), and \(C_\sigma=(C\otimes_{K,\sigma}\mathbb{C})^{\operatorname{an}}\) denotes the compact Riemann surface induced by the base change along \(\sigma\colon K\to\mathbb{C}\). Zhang showed that this inequality is implied by the more general arithmetic Standard conjecture of Gillet--Soulé \cite[Conjecture~2]{GS94}. Moreover, he proved \cite[Corollary 1.3.3]{Zha10} that equality holds in \eqref{equ_conjecture-zhang} if \(C\) is hyperelliptic. Cinkir \cite[Theorem 2.11]{Cin11} established a lower bound for the terms \(\varphi(\Gamma_v(C))\), which in particular implies \(\varphi(\Gamma_v(C))\ge 0\). In \cite[Theorem 1.2]{Wil22}, we proved that
	\[\omega_a^2\ge  \frac{g-1}{2g+1}\left(\sum_{v\in|B|}\varphi(\Gamma_v(C))\log N(v)+\sum_{\sigma\colon K\to\mathbb{C}}\varphi(C_\sigma)\right).\]
	Thus, we can summarize the above results in the following inequality:
	\begin{align}\label{equ_lowerbound-omega-varphi}
		\omega_a^2\ge\begin{cases}\frac{2g-2}{2g+1}\sum_{\sigma\colon K\to\mathbb{C}}\varphi(C_\sigma) & \text{if } C\text{ is hyperelliptic,}\\
			\frac{g-1}{2g+1}\sum_{\sigma\colon K\to\mathbb{C}}\varphi(C_\sigma) &\text{in general.}\end{cases}		
	\end{align}
	The semistability assumption can now be removed. Indeed, after passing to a finite extension \(L/K\) over which \(C\) has semistable reduction, both sides of the inequality in \eqref{equ_lowerbound-omega-varphi} scale by the same factor \([L:K]\). Dividing by this factor gives the same inequality over \(K\).
	Applying Theorem \ref{mainthm} to the inequality in \eqref{equ_lowerbound-omega-varphi}, we obtain the following corollary.
	\begin{Cor}\label{cor_lowerbound-omega}
		Let \(C\) be a smooth projective geometrically connected curve of genus \(g\ge 2\) defined over a number field \(K\) and \(\omega_a\) its canonical bundle equipped with its admissible adelic metric. We have
		\[\frac{\omega_a^2}{[K\colon \mathbb{Q}]}>\begin{cases}\frac{g(g-1)}{2g+1}(H_g-1)\ge\frac{g}{2}\log g-\frac{17}{50}g &\text{if } C \text{ is hyperelliptic,}\\ \frac{g(g+2)}{2g+1}-H_g\ge \frac{1}{2}g-\frac{13}{10}\log g&\text{in general.}\end{cases}\]
	\end{Cor}
	Next, we discuss applications to the Bogomolov conjecture. Let \(D\in \operatorname{Div}^1(C)\) be a divisor of degree \(1\) on \(C\) and write
	\[j_D\colon C\to J,\qquad x\mapsto [x-D]\]
	for the corresponding Abel--Jacobi embedding of the curve \(C\) in its Jacobian \(J=\operatorname{Pic}^0(C)\). The Bogomolov conjecture states that there is an \(\epsilon>0\) such that
	\[\#\{x\in C(\overline{K})~|~ h_{\mathrm{NT}}(j_D(x))<\epsilon\}<\infty,\]
	where \(h_{\mathrm{NT}}\) denotes the Néron--Tate height on \(J(\overline{K})\).
	We use the normalization \(h_{\mathrm{NT}}=2[K:\mathbb{Q}]\widehat h_\Theta\) as in \cite{Zha93}, where \(\widehat h_\Theta\) is the absolute Néron--Tate height associated with a symmetric theta divisor.
	Ullmo \cite{Ull98} proved this for curves, and Zhang \cite{Zha98} proved the more general statement for subvarieties of abelian varieties. Earlier, Zhang \cite[Theorem 5.6]{Zha93} already proved that
	\[\#\{x\in C(\overline{K})~|~ h_{\mathrm{NT}}(j_D(x))\le c\}<\infty\]
	for every real number \(c<\frac{\omega_a^2}{4(g-1)}\). For the definition and normalization of the Néron--Tate height \(h_{\mathrm{NT}}\) we refer to \cite{Zha93}. Thus, every explicit positive strict lower bound of \(\omega_a^2\) gives an explicit value for \(\epsilon\) in the Bogomolov conjecture. If we apply our lower bounds for \(\omega_a^2\), we obtain the following corollary.
	\begin{Cor}
		Let \(C\) be a smooth projective geometrically connected curve of genus \(g\ge 2\) defined over a number field \(K\) of degree \(d_K=[K:\mathbb{Q}]\). For every divisor \(D\in\operatorname{Div}^1(C)\) of degree \(1\) on \(C\), we have
		\[\#\left\{x\in C(\overline{K})~\left|~ h_{\mathrm{NT}}(j_D(x))\le \tfrac{d_K}{4(g-1)}\left(\tfrac{g(g+2)}{2g+1}-H_g\right) \right.\right\}<\infty.\]
		If \(C\) is hyperelliptic, then
		\[\#\left\{x\in C(\overline{K})~\left|~ h_{\mathrm{NT}}(j_D(x))\le\tfrac{d_K g}{8g+4}(H_g-1)\right.\right\}<\infty.\]
	\end{Cor}
	Thus, the Bogomolov conjecture even holds for an \(\epsilon>0\) independent of the genus \(g\), that is,
	\[\#\left\{x\in C(\overline{K})~\left|~ h_{\mathrm{NT}}(j_D(x))\le \tfrac{13}{336}d_K \right.\right\}<\infty.\]
	
	Recently, Yu, Yuan, and Zhou \cite[Theorem 1.11]{YYZ26} have proved a quantitative version of the Bogomolov conjecture. They showed that for every \(0\le \rho<\frac{1}{4g}\), we have
	\[\#\{P\in C(\overline{K})~|~h_{\mathrm{NT}}(j_D(P))\le \rho\omega_a^2\}<\frac{3.2\cdot 10^{11}g^{\frac{17}{3}}}{1-4g\rho}\left(1+10^{-6}\log\left(\frac{1}{1-4g\rho}\right)\right).\]
	Note that \(h_{\mathrm{NT}}(j_D(P))=2|P-(D-\alpha_0)|^2\) and \(\rho=2r^2\) in the notation of Yu--Yuan--Zhou. Applying our lower bound on \(\omega_a^2\), we obtain the following corollary.
	\begin{Cor}
		Let \(C\) be a smooth projective geometrically connected curve of genus \(g\ge 2\) defined over a number field \(K\) of degree \(d_K=[K:\mathbb{Q}]\). 
		For every \(\rho\in \left[0,\frac{1}{4g}\right)\), we set
		\[c(g,\rho)=\frac{3.2\cdot 10^{11}g^{\frac{17}{3}}}{1-4g\rho}\left(1+10^{-6}\log\left(\frac{1}{1-4g\rho}\right)\right).\]
		For every divisor \(D\in\operatorname{Div}^1(C)\) of degree \(1\) on \(C\), we have
		\[\#\left\{x\in C(\overline{K})~\left|~ h_{\mathrm{NT}}(j_D(x))\le d_K\rho\left(\tfrac{g(g+2)}{2g+1}-H_g\right)\right.\right\}<c(g,\rho).\]
		If \(C\) is hyperelliptic, then
		\[\#\left\{x\in C(\overline{K})~\left|~ h_{\mathrm{NT}}(j_D(x))\le \tfrac{d_K \rho g(g-1)}{2g+1}(H_g-1) \right.\right\}<c(g,\rho).\]
	\end{Cor}
	
	Next, we discuss a Northcott property satisfied by \(\omega_a^2\).
	Yuan \cite[Theorem 1.4]{Yua26} proved that there is a constant \(c_g>0\) depending only on the genus \(g\ge 2\) such that
	\begin{align}\label{equ_bound-omega-faltings}
	\omega_a^2\ge d_Kc_g\max\{h_{\mathrm{Fal}}(C),1\},
	\end{align}
	where \(h_{\mathrm{Fal}}(C)\) denotes the stable Faltings height of the curve \(C\). We refer to \cite{Yua26} for the definition and normalization of \(h_{\mathrm{Fal}}(C)\). Faltings’ \cite{Fal83} comparison of the stable Faltings height with a projective moduli height implies the Northcott property for principally polarized abelian varieties. Together with the injectivity of the Torelli map on geometric points, this implies that for every \(d,g\in\mathbb{Z}_{\geq 1}\) and \(c\in\mathbb{R}\), there are only finitely many \(\overline{\mathbb{Q}}\)-isomorphism classes of smooth projective geometrically connected curves \(C\) of genus \(g\) that are defined over a number field \(K\) with \([K:\mathbb{Q}]\leq d\) and satisfy \(h_{\mathrm{Fal}}(C)\leq c\).
	Combining our lower bound on \(\omega_a^2\) with the inequality in \eqref{equ_bound-omega-faltings}, we obtain the following completely uniform Northcott property for \(\omega_a^2\).
	\begin{Cor}
		Let \(c\in\mathbb{R}\). There are only finitely many \(\overline{\mathbb{Q}}\)-isomorphism classes of smooth projective geometrically connected curves defined over any number field, having any genus at least two, and satisfying \(\omega_a^2\le c\).
	\end{Cor}
	Note that the invariant \(\omega_a^2\) depends on the choice of a field of definition.
	\begin{proof}
		Since \(\omega_a^2\ge 0\), we may assume \(c\ge 0\).
		We set 
		\[G=\lfloor 3c+5\rfloor\qquad \text{and} \qquad \eta=\min_{2\le j\le G}c_j>0.\]
		Let \(C\) be a smooth projective geometrically connected curve defined over a number field \(K\) of degree \(d_K=[K:\mathbb{Q}]\), having genus \(g\ge 2\), and satisfying \(\omega_a^2\le c\).
		By Corollary \ref{cor_lowerbound-omega}, we have
		\[\tfrac{g-5}{3}\le \tfrac{1}{2}g-\tfrac{13}{10}\log g<\frac{\omega_a^2}{d_K}\le\omega_a^2\le c.\]
		Thus, the genus is bounded by \(2\le g\le G\). By the inequality in \eqref{equ_bound-omega-faltings}, we have
		\[d_K\le \frac{c}{\eta},\qquad h_{\mathrm{Fal}}(C)\le \frac{c}{\eta}.\]
		Thus, there are only finitely many \(\overline{\mathbb{Q}}\)-isomorphism classes as in the corollary for each fixed genus \(2\le g\le G\). Hence, there are only finitely many \(\overline{\mathbb{Q}}\)-isomorphism classes as in the corollary.
	\end{proof}
	
	\section{Analytic Arakelov invariants}\label{sec_arakelov-invariants}
	We recall the definitions of analytic invariants in Arakelov theory such as the Arakelov--Green function and the Zhang--Kawazumi invariant. Let \(X\) be a compact and connected Riemann surface of genus \(g\ge 2\). We choose a basis of one-forms \(\omega_1,\dots,\omega_g\in H^0(X,\Omega_X^1)\) which is orthonormal with respect to the inner product
	\[\langle\omega,\omega'\rangle=\frac{i}{2}\int_X\omega\wedge\overline{\omega'}.\]
	The canonical \((1,1)\)-form \(\mu\) on \(X\) is defined by
	\[\mu=\frac{i}{2g}\sum_{j=1}^g \omega_j\wedge \overline{\omega_j}.\]
	It has volume \(\int_X\mu=1\). The Arakelov--Green function \(G\colon X^2\to \mathbb{R}_{\ge 0}\) is the unique function satisfying the following properties:
	\begin{enumerate}[label=(\roman*)]
		\item The function \(G^2\) is \(C^{\infty}\) on \(X^2\).
		\item We have \(\partial_x\overline{\partial}_x\log G(x,y)^2=2\pi i(\mu(x)-\delta_y(x))\) in the sense of currents.
		\item The function is normalized by \(\int_X\log G(x,y)\mu(x)=0\).
	\end{enumerate}
	In the following we will write 
	\[g(x,y)=\log G(x,y).\]
	One can check that \(g\) is symmetric, that is, \(g(x,y)=g(y,x)\).
	
	The Laplacian operator \(\Delta\) on \(C^{\infty}\) is defined by setting
	\[\frac{1}{\pi i}\partial\overline{\partial}f=(\Delta f)\mu\]
	for every \(f\in C^{\infty}(X)\). We can also apply \(\Delta\) to a complex-valued smooth function \(f\), by applying it separately to its real part and its imaginary part:
	\[\Delta(f)=\Delta(\operatorname{Re}f+i\operatorname{Im}f)=\Delta\operatorname{Re}f+i\Delta \operatorname{Im}f.\]
	In particular, this implies that \(\Delta \overline{f}=\overline{\Delta f}\).
	
	By the above properties (i)--(iii), the Arakelov--Green function \(g(x,y)\) can be considered as an inverse of \(\Delta\) in the sense that
	\begin{align}\label{equ_faltings-invers}
	f(y)=\int_X(-g(x,y))\Delta f(x)\mu(x)
	\end{align}
	for every \(f\in C^{\infty}(X)\) satisfying \(\int_X f(x)\mu(x)=0\). For a proof we refer to \cite[Section 3]{Fal84}.
	Although \(\Delta\) is not invertible on the whole space \(C^{\infty}\) as it sends all constant functions to \(0\), we nevertheless write
	\begin{align*}
	\Delta^{-1}f(y)=-\int_Xg(x,y)f(x)\mu(x)
	\end{align*}
	for every \(f\in C^{\infty}(X)\). 
	By Equation \eqref{equ_faltings-invers}, we know that \(\Delta^{-1}\Delta f=f\) if \(\int_X f\mu=0\). The normalization \(\int_X f\mu=0\) also implies the other direction. Indeed, we have
	\begin{align*}
	\Delta(\Delta^{-1}f(y))\mu(y)&=-\frac{1}{\pi i}\partial_y\overline{\partial}_y\int_{x\in X} g(x,y)f(x)\mu(x)\\
	&=-\int_{x\in X} \frac{1}{\pi i}\partial_y\overline{\partial}_yg(x,y) f(x)\mu(x)\\
	&=-\int_{x\in X} (\mu(y)-\delta_x(y)) f(x)\mu(x)\\
	&=-\left(\int_{x\in X}f(x)\mu(x)\right)\mu(y)+\int_{x\in X}\delta_x(y)f(x)\mu(x)\\
	&=f(y)\mu(y)
	\end{align*}
	and hence \(\Delta\Delta^{-1}f=f\) if \(\int_X f\mu=0\).
	We can also apply \(\Delta^{-1}\) to complex-valued smooth functions, in which case we apply it separately to the real part and to the imaginary part. In particular, we get \(\overline{\Delta^{-1}f}=\Delta^{-1} \overline{f}\) for all complex-valued smooth functions.
	
	The Arakelov--Green function defines a Hermitian metric on the line bundle \(\mathcal{O}_{X^2}(\Delta)\) associated to the diagonal divisor \(\Delta\subseteq X^2\) by setting \(\|1_\Delta\|(x,y)=G(x,y)\), where \(1_\Delta\) denotes the canonical section of \(\mathcal{O}_{X^2}(\Delta)\). The curvature form \(h_\Delta\) of \(\mathcal{O}_{X^2}(\Delta)\) equipped with this metric is given by
	\[h_{\Delta}(x,y)=\mu(x)+\mu(y)-\frac{i}{2}\sum_{j=1}^g(\omega_j(x)\wedge \overline{\omega}_j(y)+\omega_j(y)\wedge \overline{\omega}_j(x)),\]
	see \cite[Proposition 3.1]{Ara74}. The Zhang--Kawazumi invariant \(\varphi(X)\) of \(X\) is given by
	\[\varphi(X)=\int_{X^2} g(x,y)h_{\Delta}^2(x,y),\]
	see the proof of \cite[Proposition 2.5.3]{Zha10} or \cite[Proposition 5.1]{dJo16}.
	\section{The bilinear form \texorpdfstring{\(\mathcal E\)}{E}}
	We continue the notation from the previous section. We write \(\operatorname{Herm}_g\) for the real vector space of Hermitian \(g\times g\) matrices. In this section we study the form 
	\[\mathcal{E}\colon C^{\infty}(X,\operatorname{Herm}_g)\times C^{\infty}(X,\operatorname{Herm}_g)\to\mathbb{C},\qquad (A,B)\mapsto \frac{i}{\pi}\int_X\operatorname{tr}(\partial A\wedge \overline{\partial}B).\]
	The following lemma shows that \(\mathcal{E}\) is a symmetric positive semi-definite bilinear form which is indeed real-valued.
	\begin{Lem}\label{lem_bilinear-form}
		The map \(\mathcal{E}\) defines a real-valued symmetric positive semi-definite bilinear form on \(C^{\infty}(X,\operatorname{Herm}_g)\). We have \(\mathcal{E}(A,A)=0\) if and only if \(A\) is constant on \(X\). Moreover, \(\mathcal{E}\) satisfies the identity
		\begin{align}\label{equ_E-identity}
		\mathcal{E}(A,B)=\int_X\operatorname{tr}(A\Delta B)\mu,
		\end{align}
		where the Laplacian \(\Delta\) is applied entrywise to the matrix \(B\).
	\end{Lem}
	\begin{proof}
		We first prove Equation \eqref{equ_E-identity}.
		By Stokes' theorem, we get
		\begin{align*}
			0&=\frac{i}{\pi}\int_X d\operatorname{tr}(A\overline{\partial}B)=\frac{i}{\pi}\int_X\operatorname{tr}(dA\wedge \overline{\partial}B)+\frac{i}{\pi}\int_X\operatorname{tr}(Ad\overline{\partial}B)\\
			&=\frac{i}{\pi}\int_X\operatorname{tr}(\partial A\wedge \overline{\partial}B)+\frac{i}{\pi}\int_X\operatorname{tr}(A\partial\overline{\partial}B)=\mathcal{E}(A,B)-\int_X\operatorname{tr}(A\Delta B)\mu,
		\end{align*}
		where we used that \(d=\partial+\overline{\partial}\) and that \((2,0)\)- and \((0,2)\)-forms vanish on the \(1\)-dimensional variety \(X\). This proves Equation \eqref{equ_E-identity}.
		
		Using that \(A,B\in C^{\infty}(X,\operatorname{Herm}_g)\) are Hermitian matrices, that \(\overline{\Delta B}=\Delta \overline{B}\), and the cyclicity of the trace, we get by Equation \eqref{equ_E-identity} that
		\begin{align*}
		\overline{\mathcal{E}(A,B)}&=\int_X\overline{\operatorname{tr}(A\Delta B)}\mu=\int_X \operatorname{tr}(\overline{A}\Delta \overline{B})\mu=\int_X \operatorname{tr}(A^t \Delta B^t)\mu\\
		&=\int_X \operatorname{tr}((\Delta B A)^t)\mu=\int_X \operatorname{tr}(\Delta B A)\mu=\int_X \operatorname{tr}(A\Delta B)\mu=\mathcal{E}(A,B).
		\end{align*}
		Thus, \(\mathcal{E}\) is real-valued.
		As in the proof of Equation \eqref{equ_E-identity}, we can use Stokes' theorem to obtain
		\begin{align*}
			0&=\int_X d\operatorname{tr}(AdB)=\int_X \operatorname{tr}(dA\wedge dB)=\int_X\operatorname{tr}(\partial A\wedge \overline{\partial}B)+\int_X\operatorname{tr}(\overline{\partial} A\wedge \partial B)\\
			&=\int_X\operatorname{tr}(\partial A\wedge \overline{\partial}B)-\int_X\operatorname{tr}(\partial B\wedge \overline{\partial}A)=\frac{\pi}{i}\left(\mathcal{E}(A,B)-\mathcal{E}(B,A)\right).
		\end{align*}
		 Thus, \(\mathcal{E}\) is symmetric.
		 The bilinearity of \(\mathcal{E}\) follows directly from the definition. 
		 
		 It remains to show that \(\mathcal{E}(A,A)\) is non-negative and that it is zero if and only if \(A\) is constant. We fix an arbitrary \(A\in C^{\infty}(X,\operatorname{Herm}_g)\) and we write \(A=(a_{jk})_{1\le j,k\le g}\) for the coefficients of \(A\). Note that \(\overline{a}_{jk}=a_{kj}\) for all \(j,k\in\{1,\dots,g\}\) since \(A\) is Hermitian. We fix an arbitrary local coordinate function \(z\colon U\to \mathbb{C}\) on any open subset \(U\subseteq X\). On \(U\) we compute
		 \begin{align*}
		 	\frac{i}{\pi}\operatorname{tr}(\partial A\wedge \overline{\partial}A)&=\frac{i}{\pi}\operatorname{tr}\left(\left(\sum_{k=1}^g\frac{\partial a_{jk}}{\partial z}\cdot\frac{\partial a_{kl}}{\partial \overline{z}}\cdot dz\wedge d\overline{z}\right)_{1\le j,l\le g}\right)\\
		 	&=\sum_{j,k=1}^g\frac{i}{\pi}\frac{\partial a_{jk}}{\partial z}\cdot\frac{\partial a_{kj}}{\partial \overline{z}}\cdot dz\wedge d\overline{z}\\
		 	&=\sum_{j,k=1}^g\frac{i}{\pi} \frac{\partial a_{jk}}{\partial z}\cdot\frac{\partial \overline{a}_{jk}}{\partial \overline{z}}\cdot dz\wedge d\overline{z}=\sum_{j,k=1}^g \frac{i}{\pi}\left|\frac{\partial a_{jk}}{\partial z}\right|^2 dz\wedge d\overline{z}\\
			&=\sum_{j,k=1}^g \frac{2}{\pi} \left|\frac{\partial a_{jk}}{\partial z}\right|^2 dx\wedge dy,
		 \end{align*}
		 where we used the real coordinates \(z=x+iy\) in the last step. Thus, \(\mathcal{E}(A,A)\) is the integral of a non-negative form and hence it is non-negative. We get \(\mathcal{E}(A,A)=0\) if and only if \(\frac{\partial a_{jk}}{\partial z}=0\) for all \(j,k\in\{1,\dots,g\}\) and all coordinate functions. Since \(A\) is Hermitian, this implies
		 \[\frac{\partial a_{jk}}{\partial \overline{z}}=\overline{\left(\frac{\partial \overline{a}_{jk}}{\partial z}\right)}=\overline{\left(\frac{\partial a_{kj}}{\partial z}\right)}=0.\]
		 Thus, we have \(\mathcal{E}(A,A)=0\) if and only if \(A\) is constant. This completes the proof of the lemma.
	\end{proof}
	
	Next, we connect the bilinear form \(\mathcal{E}\) to the Zhang--Kawazumi invariant \(\varphi(X)\). We fix an arbitrary local coordinate function \(z\colon U\to\mathbb{C}\) on any open subset \(U\subseteq X\). Let \(f_1,\dots,f_g\) be the holomorphic functions on \(U\) such that \(\omega_j=f_jdz\) for all \(j\in\{1,\dots,g\}\). We write \(f=(f_1,\dots,f_g)^t\) and we write \(\|f\|^2=\overline{f}^tf\). Note that \(\|f\|^2\) is a nowhere vanishing function since the line bundle \(\Omega_X^1\) is base-point-free.
	We define the matrix
	\[P=\frac{f\overline{f}^t}{\|f\|^2}.\]
	This is independent of the choice of the coordinate function \(z\). Indeed, any change of the coordinate function \(z\) changes each \(f_j\) by the same nowhere vanishing holomorphic function \(a\) on \(U\), which cancels out in the definition of \(P\). In particular, we can glue these canonical definitions of \(P\) together to a global matrix \(P\in C^{\infty}(X,\operatorname{Herm}_g)\). The following lemma gives an alternative expression for the Zhang--Kawazumi invariant in terms of the bilinear form \(\mathcal{E}\).
	\begin{Lem}\label{lem_zkE}
		Set \(F=gP-I_g\).
		\begin{enumerate}[label=(\alph*)]
			\item The matrix \(V=\Delta^{-1}F\) is Hermitian.
			\item We have \(\int_X F\mu=0\) and hence \(F=\Delta\Delta^{-1}F=\Delta V\).
			\item We have \(\varphi(X)=2\mathcal{E}(V,V).\)
		\end{enumerate}
	\end{Lem}
	\begin{proof}
		The matrix \(P\) is by construction Hermitian. Thus, the matrix \(F=gP-I_g\) is Hermitian, too. Since \(\Delta^{-1}\overline{f}=\overline{\Delta^{-1}f}\) for any complex-valued smooth function \(f\) on \(X\), we conclude that also \(V=\Delta^{-1}F\) is Hermitian. This shows (a).
		
		To prove (b), we first claim that
		\begin{align}\label{equ_P-omega}
			gP_{jk}\mu=\frac{i}{2}\omega_j\wedge \overline{\omega}_k.
		\end{align}
		To prove this, we fix again an arbitrary coordinate function \(z\colon U\to \mathbb{C}\) on an open subset \(U\subseteq X\). Using the notation from above, we get on \(U\) that
		\[g P_{jk} \mu=g\frac{f_j \overline{f}_k}{\|f\|^2}\frac{i}{2g}\sum_{l=1}^g f_l\overline{f}_ldz\wedge d\overline{z}=\frac{i}{2}f_j\overline{f}_kdz\wedge d\overline{z}=\frac{i}{2}\omega_j\wedge\overline{\omega}_k.\]
		Gluing together, we obtain Equation \eqref{equ_P-omega} globally.
		By Equation \eqref{equ_P-omega} we check that
		\begin{align}\label{equ_int-F}
			\int_X F_{jk}\mu=\frac{i}{2}\int_X\omega_j\wedge\overline{\omega}_k-\delta_{jk}=0,
		\end{align}
		where the second equality follows since the basis \(\omega_1,\dots,\omega_g\) is orthonormal. Thus, \(\int_X F\mu=0\). The equality \(F=\Delta \Delta^{-1}F\) follows from the discussion above the lemma. This proves (b).
		
		Finally, we prove (c). Using Equation \eqref{equ_P-omega}, we can rewrite the form \(h_{\Delta}^2\) as 
		\begin{align}\label{equ_computation-h2}
			h_{\Delta}^2(x,y)&=2\mu(x)\mu(y)+\frac{1}{2}\sum_{j,k=1}^g\omega_j(x)\wedge\overline{\omega}_k(x)\wedge \omega_k(y)\wedge \overline{\omega}_j(y)\\
			&=2\mu(x)\mu(y)-2g^2\sum_{j,k=1}^g P_{jk}(x)P_{kj}(y)\mu(x)\mu(y)\nonumber\\
			&=2(1-g^2\operatorname{tr}(P(x)P(y)))\mu(x)\mu(y).\nonumber
		\end{align}
		By the definition of \(P\), we have \(\operatorname{tr}(P)=1\). It follows that 
		\[\operatorname{tr}(F(x)F(y))=g^2\operatorname{tr}(P(x)P(y))-g.\]
		Applying this to the computation in \eqref{equ_computation-h2}, we obtain
		\begin{align*}
			h_{\Delta}^2(x,y)=2(1-g-\operatorname{tr}(F(x)F(y)))\mu(x)\mu(y).
		\end{align*}
		Since \(\int_{X^2}g(x,y)\mu(x)\mu(y)=0\), we get
		\begin{align}\label{equ_computation-varphi}
			\varphi(X)&=\int_{X^2}g(x,y)h_{\Delta}^2(x,y)=-2\int_{X^2}g(x,y)\operatorname{tr}(F(x)F(y))\mu(x)\mu(y)\\
			&=2\int_X \operatorname{tr}\left(\left(-\int_X g(x,y)F(x)\mu(x)\right)F(y)\right)\mu(y)\nonumber\\
			&=2\int_X \operatorname{tr}\left(\Delta^{-1} F(y) F(y)\right)\mu(y).\nonumber
		\end{align}
		By part (b), the matrix \(F\) satisfies \(F=\Delta\Delta^{-1} F=\Delta V\). Applying this to the computation in \eqref{equ_computation-varphi}, we finally obtain
		\[\varphi(X)=2\int_X\operatorname{tr}(V\Delta V)\mu=2\mathcal{E}(V,V),\]
		where the last equality follows from Lemma \ref{lem_bilinear-form}. This completes the proof of the lemma.
	\end{proof}

	\section{The osculating flag}\label{sec_osculating-flag}
	We recall the construction of the osculating flag. We continue the notation from the previous sections.
	
	We again choose an arbitrary local coordinate function \(z\colon U\to\mathbb{C}\) on an open subset \(U\subseteq X\) and holomorphic functions \(f_j\) on \(U\) such that \(\omega_j=f_jdz\) for all \(j\in\{1,\dots,g\}\). We write \(f_j^{(k)}=\frac{d^kf_j}{dz^k}\) for \(k\in\mathbb{Z}_{\ge 0}\) and \(j\in\{1,\dots,g\}\). We define the vectors \(f^{(k)}=\left(f_1^{(k)},\dots,f_g^{(k)}\right)^t\) for every \(k\in\{0,\dots,g-1\}\). Further, we denote the Wronskian by
	\[W=\det\left(f_j^{(k-1)}\right)_{1\le j,k\le g}.\]
	Since \(\omega_1,\dots,\omega_g\) form a basis of \(H^0(X,\Omega_X^1)\), the functions \(f_1,\dots,f_g\) are linearly independent over \(\mathbb{C}\). By the Wronskian criterion \cite[Lemma VII.4.4]{Mir95}, \(W\) is not identically zero on \(U\).
	Thus, the vectors \(f(x),f^{(1)}(x),\dots,f^{(g-1)}(x)\) are linearly independent over a dense open subset of \(U\). The vanishing locus of \(W\) is independent of the choice of the coordinate function. Hence, we can define the dense open subset \(X^\circ\subseteq X\) of points where the Wronskian does not vanish.
	
	For every \(k\in\{1,\dots,g\}\), we write 
	\begin{align*}
	\mathcal{F}_k=\mathcal{O}_Uf+\mathcal{O}_Uf^{(1)}+\dots+\mathcal{O}_Uf^{(k-1)}\subseteq \mathcal{O}_U^{\oplus g}
	\end{align*}
	for the coherent \(\mathcal{O}_U\)-submodule generated by the vector-valued functions \(f\), \(f^{(1)}\), \dots, \(f^{(k-1)}\). It has generic rank \(k\). It follows from the definition that
	\begin{align}\label{equ_Fk-derivative}
	\frac{d}{dz}(\mathcal{F}_k)\subseteq \mathcal{F}_{k+1}
	\end{align}
	for all \(k\in\{1,\dots,g-1\}\).	
	If we replace the local coordinate function \(z\) by another local coordinate \(\widetilde{z}\), we have to replace the functions \(f_j\) by \(\widetilde{f}_j=af_j\), where \(a=\frac{dz}{d\widetilde{z}}\) is a holomorphic function on \(U\). We use the notation \(\widetilde{f}^{(k)}_j=\frac{d^k\widetilde{f}_j}{d\widetilde{z}^k}\) and \(\widetilde{f}^{(k)}=\left(\widetilde{f}_1^{(k)},\dots,\widetilde{f}_g^{(k)}\right)^t\) for all \(j\in\{1,\dots,g\}\) and \(k\in\{0,\dots,g-1\}\). Since \(\frac{d}{d\widetilde{z}}=a \frac{d}{dz}\), we obtain
	\[\frac{d}{d\widetilde{z}}(\mathcal{F}_k)\subseteq \mathcal{F}_{k+1}\]
	for all \(k\in\{1,\dots,g-1\}\). Since \(\widetilde{f}=af\in \mathcal{F}_1\), we conclude that \(\widetilde{f}^{(k)}\in \mathcal{F}_l\) for all \(0\le k<l\le g\). This implies
	\[\widetilde{\mathcal{F}}_k=\mathcal{O}_U\widetilde{f}+\mathcal{O}_U \widetilde{f}^{(1)}+\dots+\mathcal{O}_U\widetilde{f}^{(k-1)}\subseteq \mathcal{F}_k\]
	for all \(k\in\{1,\dots,g\}\). By symmetry, we finally obtain \(\mathcal{F}_k=\widetilde{\mathcal{F}}_k\). That is, the \(\mathcal{O}_U\)-submodules \(\mathcal{F}_k\) do not depend on the choice of the local coordinate function and we can glue them together to global coherent subsheaves \(\mathcal{F}_k\subseteq \mathcal{O}_X^{\oplus g}\) for all \(k\in\{1,\dots,g\}\).
	
	The sheaves \(\mathcal{F}_k\) need not define subbundles at the zeros of the Wronskian. To address this issue, we take the saturation \(E_k=\operatorname{Sat}_{\mathcal{O}_X^{\oplus g}}\mathcal{F}_k\) of \(\mathcal{F}_k\) in \(\mathcal{O}_X^{\oplus g}\). Locally at a point \(p\in X\), we can describe \(E_k\) by its stalks
	\[E_{k,p}=\left\{ s\in \mathcal{O}_{X,p}^{\oplus g}~|~ \exists r\in\mathcal{O}_{X,p}\setminus\{0\} \text{ such that } rs\in (\mathcal{F}_k)_p\right\}.\]
	This is the minimal coherent subsheaf \(E_k\subseteq \mathcal{O}_{X}^{\oplus g}\) containing \(\mathcal{F}_k\) such that the quotient \(\mathcal{O}_X^{\oplus g}/E_k\) is torsion-free.
	By \cite[\href{https://stacks.math.columbia.edu/tag/0CC4}{Tag 0CC4}]{Stacks}, a coherent \(\mathcal{O}_X\)-module is torsion-free if and only if it is finite locally free. Thus, \(E_k\) is the smallest subbundle of \(\mathcal{O}_X^{\oplus g}\) containing \(\mathcal{F}_k\).
	The resulting flag of subbundles
	\[0\subsetneq E_1\subsetneq E_2\subsetneq\dots\subsetneq E_g=\mathcal{O}_X^{\oplus g}\]
	is called the \emph{osculating flag} associated with the basis \(\omega_1,\dots,\omega_g\in H^0(X,\Omega_X^1)\).
	
	Next, we discuss the projectors onto the subbundles \(E_k\). At every point \(x\in X\) and for every \(k\in\{1,\dots,g\}\), we can choose a neighborhood \(U\) of \(x\) and sections
	\(s_{k,1},\dots,s_{k,k}\in H^0\left(U,\mathcal{O}_X^{\oplus g}\right)\)
	which form an \(\mathcal{O}_U\)-basis of \(E_k|_U\).  We call the \(g\times k\) matrix \(S_k=\left(s_{k,1},\dots,s_{k,k}\right)\) a local frame of \(E_k\). 
	Since the columns of \(S_k\) are linearly independent, the matrix \(G_k=\overline{S}_k^tS_k\) is positive definite and therefore invertible.
	We define the matrix
	\[P_k=S_kG_k^{-1}\overline{S}_k^t.\]
	One checks directly that \(P_k\) is Hermitian.
	If we choose a different basis \[s'_{k,1},\dots,s'_{k,k}\in H^0(U,\mathcal{O}_X^{\oplus g}),\] the matrix \(S_k\) changes by an everywhere invertible \(k\times k\) matrix \(A\) to \(S_k'=S_k\cdot A\). In the definition of \(P_k\) the factor \(A\) cancels out. Thus, the matrix \(P_k\) is independent of the choice of the local frame and we can glue it to a global matrix \(P_k\in C^{\infty}(X,\operatorname{Herm}_g)\). By construction, the matrix \(P_k\) is the orthogonal projector onto the subbundle \(E_k\) with respect to the standard Hermitian metric on \(\mathcal{O}_X^{\oplus g}\). Since \(f\) is a local frame of \(E_1\), we obtain that \(P_1\) coincides with the matrix \(P\) from the previous section.
	
	We conclude this section by constructing some canonical sections of the line bundles \(\det(E_k)\otimes (\Omega_X^1)^{\otimes k(k+1)/2}\) which will be needed later. With the notation as above, we define locally
	\[\mathcal{W}_k=\left(f\wedge f^{(1)}\wedge\dots\wedge f^{(k-1)}\right)(dz)^{k(k+1)/2}\in H^0\left(U,\det E_k\otimes (\Omega_X^1)^{\otimes k(k+1)/2}\right)\]
	for every \(k\in\{1,\dots,g\}\).
	One checks that \(\mathcal{W}_k\) is independent of the choice of the local coordinate function. Indeed, by induction one shows that for another local coordinate \(\widetilde{z}\) we get
	\[\widetilde{f}^{(r)}=a^{r+1}f^{(r)}+\sum_{s=0}^{r-1}b_{rs}f^{(s)},\]
	for all \(r\in\{1,\dots,g-1\}\), where the \(b_{rs}\)'s are suitable holomorphic functions on \(U\). Since only the top terms are relevant for the wedge product in \(\mathcal{W}_k\), it changes by
	\[\widetilde{f}\wedge \widetilde{f}^{(1)}\wedge\dots\wedge \widetilde{f}^{(k-1)}=a^{k(k+1)/2}\cdot \left(f\wedge f^{(1)}\wedge\dots\wedge f^{(k-1)}\right).\]
	On the other hand, \((dz)^{k(k+1)/2}\) changes to \[(d\widetilde{z})^{k(k+1)/2}=a^{-k(k+1)/2}(dz)^{k(k+1)/2}.\]
	Thus, we obtain a global section 
	\[\mathcal{W}_k\in H^0\left(X,\det E_k\otimes (\Omega_X^1)^{\otimes k(k+1)/2}\right)\setminus\{0\}\]
	which is non-zero since the Wronskian does not vanish identically.
	
	\section{The bilinear form \texorpdfstring{\(\mathcal E\)}{E} on the projector matrices}
	We compute the values of the bilinear form \(\mathcal{E}\) on pairs of the projector matrices \(P_j\) and the matrix \(V\). We continue the notation from the previous sections. The following proposition is the main result of this section.
	\begin{Pro}\label{pro_E-projectors}
			Let \(k,l\in\{1,\dots,g\}\) satisfy \(k \neq l\). We have
			\begin{enumerate}[label=(\alph*)]
				\item \(\mathcal{E}(P_k,P_l)=0\),
				\item \(\mathcal{E}(P_k,P_k)=-2\deg(E_k)\),
				\item \(\operatorname{tr}(P_k\Delta V)=g-k\) and hence \(\mathcal{E}(P_k,V)=g-k\).
			\end{enumerate}
	\end{Pro}
	In the proof we will use the following two auxiliary lemmas. The first lemma gives some identities of the derivatives of the projectors \(P_k\).
	\begin{Lem}\label{lem_aux} We have the following identities:
		\begin{enumerate}[label=(\alph*)]
			\item For every \(1\le k\le g\) we have \(P_k(\partial P_k)=0\) and \(\partial P_k=(\partial P_k)P_k\).
			\item For every \(1\le k<l\le g\), we have \(\partial P_k=P_l(\partial P_k)P_l\).
		\end{enumerate}
	\end{Lem}
	\begin{proof}
		Since \(P_k\) is the projector onto \(E_k\), every local holomorphic section \(s\in H^0(U,E_k)\) satisfies \(P_ks=s\), and hence
		\[(\overline{\partial}P_k)s=\overline{\partial}(P_ks)-P_k\overline{\partial}s=0.\]
		We conclude that \((\overline{\partial}P_k)P_k=0\). This implies the first formula of (a) since
		\[P_k(\partial P_k)=(\overline{(\overline{\partial} P_k)P_k})^t=0,\]
		where we have used that \(P_k\) is Hermitian.
		Since \(P_k\) is a projector, we get
		\begin{align}\label{equ_PP-P}
		\partial P_k=\partial \left(P_k^2\right)=(\partial P_k)P_k+P_k(\partial P_k)=(\partial P_k)P_k,
		\end{align}
		the second formula of (a). This completes the proof of part (a).
		
		Next, we prove (b). The matrix \(\partial P_k\) acts on a local holomorphic section \(s\in H^0(U,E_k)\) by
		\begin{align}\label{equ_partial-P-s}
			(\partial P_k)s=\partial(P_ks)-P_k\partial s=(I_g-P_k)\partial s.
		\end{align}		
		Let us show that \(\partial s\in H^0(U,E_{k+1}\otimes \Omega_X^1)\).
		We have \(\mathcal{F}_k|_{X^{\circ}}=E_k|_{X^{\circ}}\) on \(X^\circ\), the non-vanishing locus of the Wronskian. From the inclusion in \eqref{equ_Fk-derivative} we obtain that
		\[\partial s|_{U\cap X^{\circ}}\in H^0(U\cap X^{\circ},E_{k+1}\otimes \Omega_X^1)\]
		for all local holomorphic sections \(s\in H^0(U,E_k)\). Thus, the image of the section \(\partial s\) in the bundle \((\mathcal{O}_X^{\oplus g}/E_{k+1})\otimes \Omega_X^1\) vanishes on the dense open subset \(U\cap X^{\circ}\) of \(U\). Hence, it vanishes identically on \(U\) and we obtain
		\[\partial s\in H^0(U,E_{k+1}\otimes \Omega_X^1).\]		
		
		Applying this to Equation \eqref{equ_partial-P-s}, we conclude that the image of the restriction of \(\partial P_k\) to \(E_k\) lies in \(E_{k+1}\).
		Combining this with Equation \eqref{equ_PP-P}, we see that the image of \(\partial P_k\) on the whole bundle \(\mathcal O_X^{\oplus g}\) lies in \(E_{k+1}\), and hence		
		\begin{align}\label{equ_Plk}
		P_{l}(\partial P_k)=\partial P_k
		\end{align}
		since \(l\ge k+1\).
		Since \(E_k\subseteq E_l\), we have \(P_kP_l=P_k\). Together with Equations \eqref{equ_PP-P} and \eqref{equ_Plk} this yields
		\[\partial P_k=(\partial P_k) P_k=(\partial P_k)P_k P_l=(\partial P_k) P_l=P_l(\partial P_k) P_l\]
		for all \(k<l\le g\). This proves (b).
	\end{proof}	
	The second auxiliary lemma connects the Chern form \(c_1(\det E_k)\) to the trace \(\mathrm{tr}(\partial P_k\wedge \overline{\partial}P_k)\). In the following, we will consider \(c_1(\det E_k)\) as a differential form by equipping \(\det E_k\) with the Hermitian metric induced by the standard Hermitian metric on \(\mathcal{O}_X^{\oplus g}\).
	\begin{Lem}\label{lem_c1}
		Let \(k\in\{1,\dots,g\}\). With the notation as above, we have the following identities:
		\begin{enumerate}[label=(\alph*)]
			\item \(\partial P_k=(I_g-P_k)(\partial S_k) G_k^{-1}\overline{S}_k^t\),
			\item \(c_1(\det E_k)=-\frac{i}{2\pi}\mathrm{tr}(\partial P_k\wedge \overline{\partial}P_k)\).
		\end{enumerate}
	\end{Lem}
	\begin{proof}
		We first prove (a). By the definition of \(P_k\), we have \(\overline{S}_k^tP_k=\overline{S}_k^t\). Thus,
		\[0=\partial(\overline{S}_k^tP_k)=\overline{S}_k^t(\partial P_k).\]
		By the definition of \(P_k\) we also have \(P_kS_k=S_k\). Thus,
		\[(\partial P_k) S_k=\partial (P_k S_k)-P_k(\partial S_k)=(I_g-P_k)\partial S_k.\]
		Applying Lemma \ref{lem_aux} (a), we obtain
		\[\partial P_k=(\partial P_k) P_k=(\partial P_k) S_k G_k^{-1} \overline{S}_k^t=(I_g-P_k)(\partial S_k) G_k^{-1}\overline{S}_k^t.\]
		This shows (a).
		
		To prove (b), we also fix a local frame \(S_k\) of \(E_k\).
		The wedge product \(s_{k,1}\wedge\dots\wedge s_{k,k}\) induced by the local frame \(S_k=(s_{k,1}, \dots, s_{k,k})\) is a local holomorphic frame of the line bundle \(\det E_k\). Its norm is given by
		\[\|s_{k,1}\wedge\dots\wedge s_{k,k}\|^2=\det(G_k),\]
		where \(G_k=\overline{S}_k^t S_k\).
		Thus, the curvature formula gives
		\[c_1(\det E_k)=-\frac{i}{2\pi}\partial\overline{\partial}\log \det(G_k)=\frac{i}{2\pi}\overline{\partial}\partial\log \det(G_k)\]
		locally where \(S_k\) is defined.
		By Jacobi's formula, we get
		\[\partial \log \det(G_k)=\frac{\partial\det(G_k)}{\det(G_k)}=\frac{\det G_k\operatorname{tr}(G_k^{-1}(\partial G_k))}{\det G_k}=\operatorname{tr}(G_k^{-1}(\partial G_k)).\]
		
		To compute \(\overline{\partial}\operatorname{tr}(G_k^{-1}\partial G_k)\), we have to study the derivatives of \(G_k\) and \(G_k^{-1}\) in more detail. Differentiating \(G_kG_k^{-1}=I_k\) gives \(G_k(\overline{\partial}G_k^{-1})=-(\overline{\partial}G_k)G_k^{-1}\) and hence
		\[\overline{\partial} G_k^{-1}=-G_k^{-1}(\overline{\partial}G_k)G_k^{-1}.\]
		Since \(S_k\) is holomorphic, we have \(\overline{\partial}S_k=0\) and \(\partial \overline{S}_k^t=0\). For simplicity, we omit wedge symbols in products of matrix-valued differential forms.
		Since \(G_k=\overline{S}_k^t S_k\), it follows that
		\[\partial G_k=\overline{S}_k^t(\partial S_k),\qquad \overline{\partial}G_k=\overline{(\partial S_k)}^tS_k,\qquad \overline{\partial}\partial G_k=\overline{\partial}\left(\overline{S}_k^t(\partial S_k)\right)=\overline{\left(\partial S_k\right)}^t(\partial S_k).\]
		We can now compute that
		\begin{align*}
			\overline{\partial}\operatorname{tr}(G_k^{-1}(\partial G_k))&=\operatorname{tr}\left(G_k^{-1}(\overline{\partial}\partial G_k)+(\overline{\partial}G_k^{-1})(\partial G_k)\right)\\
			&=\operatorname{tr}\left(G_k^{-1}\overline{(\partial S_k)}^t (\partial S_k)-G_k^{-1}\overline{(\partial S_k)}^tS_kG_k^{-1} \overline{S_k}^t(\partial S_k)\right)\\
			&=\operatorname{tr}\left(G_k^{-1}\overline{(\partial S_k)}^t(I_g-P_k)(\partial S_k)\right).
		\end{align*}
		
		Next, we connect this expression to the derivatives of the projector \(P_k\).		
		Taking the conjugate transpose of the formula in part (a), we get
		\[\overline{\partial}P_k=\overline{(\partial P_k)}^t=\overline{((I_g-P_k)(\partial S_k)G_k^{-1}\overline{S}_k^t)}^t=S_k G_k^{-1}\overline{(\partial S_k)}^t(I_g-P_k)\]
		using that \(P_k\), \((I_g-P_k)\), and \(G_k^{-1}\) are Hermitian matrices. Using the graded cyclicity of the trace and the identity \((I_g-P_k)^2=(I_g-P_k)\), we compute
		\begin{align*}
			\operatorname{tr}(\partial P_k\wedge\overline{\partial}P_k)&=\operatorname{tr}((I_g-P_k)(\partial S_k) G_k^{-1}\overline{S}_k^tS_k G_k^{-1}\overline{(\partial S_k)}^t(I_g-P_k))\\
			&=-\operatorname{tr}(G_k^{-1}\overline{S}_k^tS_k G_k^{-1}\overline{(\partial S_k)}^t(I_g-P_k)(\partial S_k))\\
			&=-\operatorname{tr}(G_k^{-1}\overline{(\partial S_k)}^t(I_g-P_k)(\partial S_k))=-\overline{\partial}\operatorname{tr}(G_k^{-1}(\partial G_k)).
		\end{align*}
		Putting everything together, we conclude
		\begin{align*}
			c_1(\det E_k)=\frac{i}{2\pi}\overline{\partial}\partial\log\det(G_k)=\frac{i}{2\pi}\overline{\partial}\operatorname{tr}(G_k^{-1}(\partial G_k))=-\frac{i}{2\pi}\operatorname{tr}(\partial P_k\wedge\overline{\partial} P_k).
		\end{align*}
		This completes the proof of the lemma.
	\end{proof}
	
	We can now give the proof of Proposition \ref{pro_E-projectors}.
	\begin{proof}[Proof of Proposition \ref{pro_E-projectors}]
		We first prove (a). By the symmetry of \(\mathcal{E}\), we may assume that \(k<l\).
		Using Lemma \ref{lem_aux} (b) and the cyclicity of the trace, we obtain
		\[\operatorname{tr}(\partial P_k\wedge\overline{\partial}P_l)=\operatorname{tr}(P_l(\partial P_k)P_l\wedge\overline{\partial}P_l)=\operatorname{tr}((\partial P_k)P_l\wedge(\overline{\partial}P_l)P_l)=0.\]
		Integrating over \(X\) gives \(\mathcal{E}(P_k,P_l)=0\). This proves (a).
		Part (b) follows directly from Lemma \ref{lem_c1} since
		\[\mathcal{E}(P_k,P_k)=\frac{i}{\pi}\int_X\operatorname{tr}(\partial P_k\wedge \overline{\partial}P_k)=-2\int_X c_1(\det E_k)=-2\deg(E_k).\]
		
		Now we prove (c). By Lemma \ref{lem_zkE} (b), we have
		\[\operatorname{tr}(P_k\Delta V)=\operatorname{tr}(P_k(gP_1-I_g))=g\operatorname{tr}(P_1)-\operatorname{tr}(P_k),\]
		where \(P_kP_1=P_1\) follows from \(E_1\subseteq E_k\). Since \(P_j\) is the projector matrix onto the subbundle \(E_j\) of rank \(j\) in \(\mathcal{O}_X^{\oplus g}\), we have \(\operatorname{tr}(P_j)=j\) for all \(j\in\{1,\dots,g\}\). Applying this to the above formula, we obtain the first statement
		\[\operatorname{tr}(P_kF)=g-k.\]		
		By Lemma \ref{lem_bilinear-form} and the identity \(\Delta V=F\), we have
		\[\mathcal{E}(P_k,V)=\int_X\operatorname{tr}(P_kF)\mu=(g-k)\int_X \mu=g-k.\]
		This shows the second statement of (c). 
	\end{proof}
	
	\section{The main result for general curves}
	We give the proof of the first statement of Theorem \ref{mainthm} in the case of general curves. We continue the notation from the previous sections. We first prove the following lemma which gives a lower bound of the Zhang--Kawazumi invariant in terms of the degrees \(\deg E_k\).
	\begin{Lem}\label{lem_lower-bound}
		We have \(\deg E_k<0\) for all \(k\in\{1,\dots,g-1\}\) and
		\[\varphi(X)> \sum_{k=1}^{g-1}\frac{(g-k)^2}{(-\deg E_k)}.\]
	\end{Lem}
	The idea of the proof is to apply the positive semi-definite form \(\mathcal{E}\) to a matrix of the form \(V+\sum_{k=1}^g a_k P_k\). Let us first check in a separate lemma that this matrix is not constant on \(X\).
	\begin{Lem}\label{lem_not-constant}
		For all \(a\in \mathbb{R}^g\) the matrix \(V+\sum_{k=1}^g a_k P_k\) is not constant on \(X\).
	\end{Lem}
	\begin{proof}
		 Let us assume that the matrix \(V+\sum_{k=1}^g a_k P_k\) is constant. Then we obtain
		\[\operatorname{tr}(P_l \partial V)=-\sum_{k=1}^{g}a_k\operatorname{tr}(P_l\partial P_k)\]
		for every \(l\in\{1,\dots, g\}\).
		By Lemma \ref{lem_aux} (a), by the cyclicity of the trace, and by the commutativity of the projectors \(P_k\) and \(P_l\), we get
		\begin{align*}
			\operatorname{tr}(P_l\partial P_k)=\operatorname{tr}(P_l(\partial P_k)P_k)=\operatorname{tr}(P_lP_k(\partial P_k))=0.
		\end{align*}
		Thus, \(\operatorname{tr}(P_l\partial V)=0\) for all \(l\in\{1,\dots,g\}\).
				
		Since \(\det E_g=\det\mathcal{O}_X^{\oplus g}=\mathcal{O}_X\) is trivial, the section \(\mathcal{W}_g\) is a non-zero section of the ample line bundle \((\Omega_X^1)^{\otimes g(g+1)/2}\). Thus, we can choose a point \(p\in X\) such that \(\mathcal{W}_g(p)=0\). Since \(\mathcal{W}_1(p)=(\omega_1(p),\dots,\omega_g(p))^t\neq 0\), we can choose an \(l\in\{1,\dots,g-1\}\) such that
		\[\mathcal{W}_l(p)\neq 0,\qquad \mathcal{W}_{l+1}(p)=0.\]
		Locally at \(p\), the vectors \(f,f^{(1)},\dots,f^{(l-1)}\in E_l\) are linearly independent and hence \(S_l=(f,f^{(1)},\dots,f^{(l-1)})\) defines a local frame of \(E_l\). But in the fiber at \(p\), the vectors \(f(p),f^{(1)}(p),\dots,f^{(l)}(p)\in \mathbb{C}^g\) are linearly dependent since \(\mathcal{W}_{l+1}(p)=0\). Hence \(f^{(l)}(p)\in E_{l}(p)\) and \(\operatorname{Image}(\frac{dS_l}{dz}(p))\subseteq E_{l}(p)\), where \(E_{l}(p)=E_{l,p}\otimes_{\mathcal{O}_{X,p}}\mathbb{C}\) denotes the fiber of \(E_l\) at \(p\). We conclude that
		\[(I_g-P_l)(\partial S_l)(p)=0.\]
		Using Lemma \ref{lem_c1} (a), we conclude that
		\[(\partial P_l)(p)=((I_g-P_l)(\partial S_l) G_l^{-1}\overline{S}_l^t)(p)=0.\]
		Since \(P_l\) is Hermitian, we obtain
		\[(\overline{\partial}P_l)(p)=\overline{(\partial P_l)(p)}^t=0.\]
		
		We apply both results \(\operatorname{tr}(P_l\partial V)=0\) and \((\overline{\partial}P_l)(p)=0\) to obtain, at \(p\),
		\[0=\overline{\partial}\operatorname{tr}(P_l\partial V)=\operatorname{tr}(P_l\overline{\partial}\partial V)=-\pi i\operatorname{tr}(P_l\Delta V)\mu.\]
		Now Proposition \ref{pro_E-projectors} (c) implies \(0=-\pi i(g-l)\mu\) contradicting that \(l<g\) and \(\mu\) is a positive form. Thus, the matrix \(V+\sum_{k=1}^g a_k P_k\) cannot be constant and the lemma is proven.
	\end{proof}
	Now we give the proof of Lemma \ref{lem_lower-bound}.
	\begin{proof}[Proof of Lemma \ref{lem_lower-bound}]
		Since \(\mathcal{E}\) is positive semi-definite, Proposition \ref{pro_E-projectors} (b) implies that \(\deg E_k\le 0\) for all \(k\in\{1,\dots, g\}\). If \(\deg E_k=0\), Lemma \ref{lem_bilinear-form} and Proposition \ref{pro_E-projectors} (b) imply that the projector matrix \(P_k\) is constant on \(X\). By Proposition \ref{pro_E-projectors} (c), this would imply that
		\begin{align*}
			g-k=\int_X \operatorname{tr}(P_k(gP_1-I_g))\mu=\operatorname{tr}\left( P_k\int_X(gP_1-I_g)\mu\right)=0
		\end{align*}
		by Lemma \ref{lem_zkE} (b). Hence, this can only happen if \(g=k\). This shows the first statement of the lemma.
		
		By Lemmas \ref{lem_bilinear-form} and \ref{lem_not-constant}, we have
		\begin{align*}
			0&< \mathcal{E}\left( V+\sum_{k=1}^{g-1}\frac{g-k}{2\deg E_k}P_k,V+\sum_{k=1}^{g-1}\frac{g-k}{2\deg E_k}P_k\right)\\
			&=\mathcal{E}(V,V)+2\sum_{k=1}^{g-1}\frac{g-k}{2\deg E_k}\mathcal{E}(V,P_k)+\sum_{k,l=1}^{g-1}\frac{(g-k)(g-l)}{4\deg E_k\cdot \deg E_l}\mathcal{E}(P_k,P_l).
		\end{align*}
		Using Lemma \ref{lem_zkE} (c) and Proposition \ref{pro_E-projectors}, we deduce
		\[0< \frac{1}{2}\varphi(X)+\sum_{k=1}^{g-1}\frac{(g-k)^2}{2\deg E_k}\]
		and hence
		\[\varphi(X)>\sum_{k=1}^{g-1}\frac{(g-k)^2}{(-\deg E_k)}.\]
		This proves the lemma.
	\end{proof}
	
	We can now give the proof of the first statement of Theorem \ref{mainthm} in the case of general curves. The existence of the non-zero section \(\mathcal{W}_k\) shows that
	\[0\le \deg\left(\det E_k\otimes(\Omega_X^1)^{\otimes k(k+1)/2}\right)=\deg E_k+\tfrac{k(k+1)}{2}(2g-2).\]
	Thus, we get the upper bound
	\[-\deg E_k\le k(k+1)(g-1).\]
	If we apply it to the inequality in Lemma \ref{lem_lower-bound}, we obtain
	\begin{align*}
		\varphi(X)>\sum_{k=1}^{g-1}\frac{(g-k)^2}{k(k+1)(g-1)}.
	\end{align*}
	One checks directly that
	\[\frac{(g-k)^2}{k(k+1)}=1+\frac{g^2}{k}-\frac{(g+1)^2}{k+1}\]
	for all \(k\in\{1,\dots, g-1\}\) and hence
	\[\sum_{k=1}^{g-1}\frac{(g-k)^2}{k(k+1)}=g(g+2)-(2g+1)H_g\]
	after summing over \(1\le k\le g-1\), where \(H_g=\sum_{j=1}^g \frac{1}{j}\) denotes the \(g\)-th harmonic number. Thus, 
	\[\varphi(X)>\frac{g(g+2)-(2g+1)H_g}{g-1}\]
	as stated in the first statement of Theorem \ref{mainthm}. The lower bounds \(\frac{1}{2}\) for \(g=2\) and \(g-2\log g\) for \(g\ge 3\) can easily be checked. This proves the first statement of Theorem \ref{mainthm} in the case of general curves.
	
	\section{The main result for hyperelliptic curves}
	We prove the second statement of Theorem \ref{mainthm} in the case of hyperelliptic curves. 
	We first compute the values of \(\deg E_k\) exactly and then we will apply Lemma \ref{lem_lower-bound}.
	We use the notation from the previous sections.
	
	If we choose another basis of \(H^0(X,\Omega_X^1)\) to construct the osculating flag in Section \ref{sec_osculating-flag}, the resulting flag changes by a change of basis of \(\mathcal{O}_X^{\oplus g}\) and hence by a constant matrix in \(\operatorname{GL}_g(\mathbb{C})\). Thus, the numbers \(\deg E_k\) for \(k\in\{1,\dots,g\}\) are independent of the choice of the basis and the construction even allows us to take a non-orthonormal basis.
	
	We choose a Weierstraß equation
	\(y^2=h(x)\)
	of \(X\) such that the point at infinity is not a branch point, that is, \(\deg h=2g+2\). 
	Let \(\pi\colon X\to\mathbb{P}_{\mathbb{C}}^1\) be the hyperelliptic double cover induced by the affine map \((x,y)\mapsto x\).
	We write 
	\[U_0=\pi^{-1}\left(\{[x_0:x_1]~|~x_1\neq 0\}\right).\]
	The standard basis of \(H^0(X,\Omega_X^1)\) associated with this equation is given by
	\[\omega_1=\frac{dx}{y},~\omega_2=\frac{xdx}{y},~\dots,~\omega_g=\frac{x^{g-1}dx}{y}.\]
	We write \(U_0^{\circ}=U_0\cap\{y\neq 0\}\). On \(U_0^{\circ}\) the function \(x\) is a local coordinate and we can write \(\omega_j=f_jdx\) with \(f_j=\frac{x^{j-1}}{y}\) for \(j\in\{1,\dots,g\}\). 	We define the vector \(\rho=(1,x,\dots,x^{g-1})^t\). Then \(f=y^{-1}\rho\) on \(U_0^{\circ}\) and the Leibniz rule gives
	\[f^{(m)}=y^{-1}\rho^{(m)}+\sum_{j=0}^{m-1}a_{m,j}\rho^{(j)}\]
	for suitable functions \(a_{m,j}\) on \(U_0^{\circ}\) and \(m\in\{1,\dots,g-1\}\). Since \(y\) is invertible on \(U_0^\circ\), the corresponding triangular change of the generators \(f,f^{(1)},\dots,f^{(k-1)}\) of \(\mathcal{F}_k\) is invertible. Thus,
	\[\mathcal{F}_k|_{U_0^{\circ}}=\mathcal{O}_{U_0^{\circ}}\rho+\mathcal{O}_{U_0^{\circ}}\rho^{(1)}+\dots+\mathcal{O}_{U_0^{\circ}}\rho^{(k-1)}.\]
	A direct computation gives
	\[\rho^{(k)}=\left(0,\dots,0,k!,\frac{(k+1)!}{1!}x^1,\dots,\frac{(g-1)!}{(g-1-k)!}x^{g-1-k}\right)^t\]
	with \(k\) zeros at the beginning. Thus, the vectors \(\rho,\rho^{(1)},\dots,\rho^{(g-1)}\) are linearly independent everywhere on \(U_0\). Hence, for every \(k\in\{1,\dots,g\}\) we get a local frame \(S_k=\left(\rho,\rho^{(1)},\dots,\rho^{(k-1)}\right)\) of \(E_k\) over \(U_0^{\circ}\) which naturally extends to \(U_0\).
	
	We now consider the second standard affine chart of \(\mathbb{P}_{\mathbb{C}}^1\) and write
	\[U_\infty=\pi^{-1}\left(\{[x_0:x_1]~|~x_0\neq0\}\right).\]
	On \(U_0\cap U_\infty\), we set
	\[v=\frac{1}{x},\qquad w=\frac{y}{x^{g+1}}.\]
	Then \(w^2=v^{2g+2}h\left(\frac{1}{v}\right)\),
	so that \(v\) and \(w\) extend to regular functions on \(U_\infty\). Moreover,
	\[x=\frac{1}{v},\qquad dx=-\frac{dv}{v^2},\qquad y=\frac{w}{v^{g+1}},\]
	and therefore
	\[\omega_j=-\frac{v^{g-j}}{w}dv\]
	for every \(j\in\{1,\dots,g\}\). We define the vector \(\sigma=(v^{g-1},v^{g-2},\dots,1)^t\).
	With the same argument as above, we obtain the local frame
	\[T_k=\left(\sigma,\sigma^{(1)},\dots,\sigma^{(k-1)}\right)\]
	of \(E_k\) over \(U_\infty\), where the derivatives are taken with respect to \(v\).
	
	On \(U_0\cap U_\infty\), we have \(\sigma=v^{g-1}\rho\) and
	\(\frac{d}{dv}=-v^{-2}\frac{d}{dx}\).
	An induction on \(m\) gives
	\[\sigma^{(m)}=(-1)^mv^{g-1-2m}\rho^{(m)}+\sum_{j=0}^{m-1}b_{m,j}\rho^{(j)}\]
	for suitable regular functions \(b_{m,j}\) on \(U_0\cap U_\infty\) and \(m\in\{1,\dots,g-1\}\). Consequently, there is a triangular matrix \(A_k\) such that
	\(T_k=S_kA_k\),
	whose diagonal entries are
	\[v^{g-1},-v^{g-3},\dots,(-1)^{k-1}v^{g-2k+1}.\]
	Hence,
	\[\det A_k=(-1)^{\frac{k(k-1)}{2}}v^{k(g-k)}.\]
	Thus, the transition function of \(\det E_k\) differs by a nonzero constant from \(v^{k(g-k)}\). Since \(v^m\) is the transition function of \(\mathcal O_{\mathbb P^1}(-m)\) with respect to the standard affine charts, the above transition function yields
	\[\det E_k\simeq\pi^*\mathcal{O}_{\mathbb{P}^1}(-k(g-k)).\]
	Since \(\pi\) has degree \(2\), we conclude that
	\[\deg E_k=-2k(g-k).\]
	
	We can now apply this value to Lemma \ref{lem_lower-bound}. This yields
	\[\varphi(X)>\sum_{k=1}^{g-1}\frac{(g-k)^2}{2k(g-k)}=\sum_{k=1}^{g-1}\frac{g-k}{2k}=\frac{1}{2}(gH_{g-1}-(g-1))=\frac{g}{2}(H_g-1).\]
	The lower bound \(\frac{g}{2}(H_g-1)\ge \frac{g}{2}(\log g-\frac{1}{2})\) can be checked directly. This completes the proof of Theorem \ref{mainthm}. The proofs of the corollaries follow directly from the discussions around them.
	
\end{document}